\documentclass[12pt,a4paper]{amsart}
\usepackage{tikz-cd,mathtools,amssymb}

\theoremstyle{plain}
\newtheorem{thm}{Theorem}[section]
\newtheorem{prop}[thm]{Proposition}
\newtheorem{lem}[thm]{Lemma}

\newcommand{\C}{\mathbb C}

\newcommand{\Q}{\mathbb Q}
\newcommand{\Z}{\mathbb Z}
\newcommand{\N}{\mathbb N}
\newcommand{\PP}{\mathbb P}

\DeclareMathOperator{\Coh}{Coh}
\DeclareMathOperator{\add}{add}
\DeclareMathOperator{\mmod}{mod}
\DeclareMathOperator{\rep}{rep}
\DeclareMathOperator{\GL}{GL}
\DeclareMathOperator{\tr}{tr}
\DeclareMathOperator{\Hom}{Hom}
\DeclareMathOperator{\Ext}{Ext}
\DeclareMathOperator{\Ker}{Ker}
\DeclareMathOperator{\Ima}{Im}

\DeclareMathOperator{\rank}{rank}
\DeclareMathOperator{\Res}{Res}

\DeclareMathOperator{\BunConn}{BunConn}
\DeclareMathOperator{\CohConn}{CohConn}
\DeclareMathOperator{\Bun}{Bun}

\DeclareMathOperator{\dimv}{\underline\dim}

\newcommand{\OO}{\mathcal O}
\newcommand{\F}{\mathcal F}
\newcommand{\E}{\mathcal E}
\newcommand{\G}{\mathcal G}
\newcommand{\CC}{\mathcal C}
\newcommand{\XX}{\mathbb X}
\newcommand{\ox}{\vec x}

\newcommand{\oc}{\vec c}
\newcommand{\oom}{\vec{\omega}}
\newcommand{\bw}{\mathbf w}
\newcommand{\I}{\textit{\textbf i}\mkern2mu}

\usepackage{hyperref}

\begin{document}
\title{The Deligne--Simpson Problem}
\author{William Crawley-Boevey}
\author{Andrew Hubery}
\address{Fakult\"at f\"ur Mathematik, Universit\"at Bielefeld, 33501 Bielefeld, Germany}
\email{wcrawley@math.uni-bielefeld.de}
\email{hubery@math.uni-bielefeld.de}







\thanks{The authors are supported by the Deutsche Forschungsgemeinschaft (DFG, German Research Foundation) – Project-ID 491392403 – TRR 358, and previously were supported by the Alexander von Humboldt Foundation in the framework of an Alexander von Humboldt Professorship endowed by the German Federal Ministry of Education and Research.}

\begin{abstract}
Given $k$ similarity classes of invertible matrices, the Deligne--Simpson problem asks to determine whether or not one can find matrices in these classes whose product is the identity and with no common invariant subspace. The first author conjectured an answer in terms of an associated root system, and proved one implication in joint work with Shaw. In this paper we prove the other implication, thus confirming the conjecture.
\end{abstract}
\maketitle

\section{Introduction}
Given conjugacy classes $C_1,\dots,C_k$ in $\GL_n(\C)$, the Deligne--Simpson problem asks to determine whether or not there is a solution to the equation
\[ A_1 A_2 \cdots A_k = 1 \qquad (A_i \in C_i) \]
which is irreducible in the sense that the $A_i$ have no non-trivial common invariant subspace. This is motivated by the problem of classifying systems of linear ordinary differential equations in the complex domain in terms of their local monodromies; see \cite{Kostovsurvey} for more background and motivation. In~\cite{CBipb} the first author has given a conjectural solution to this problem, and one implication has been proven by Crawley-Boevey and Shaw~\cite{CBSh}. Here we prove the other implication, thus confirming the conjecture.

In order to fix the conjugacy classes, we observe that they are determined by a \emph{weight sequence} $\bw=(w_1,\dots,w_k)$ of positive integers, a collection of non-zero complex numbers $\xi = (\xi_{ip})$ ($1\le i\le k$, $1\le p\le w_i$), and integer sequences $n_{ip}$ ($1\leq i\leq k$, $0\leq p\leq w_i$) with 
\[ n = n_{i0} \geq n_{i1} \geq \cdots \geq n_{iw_i-1} \geq n_{iw_i} = 0. \]
The correspondence is such that $A_i\in\ C_i$ if and only if 
\[ \rank(A_i-\xi_{ip})\cdots(A_i-\xi_{i1}) = n_{ip} \quad (0\leq p\leq w_i). \]
For example $w_i$ can be the degree of the minimal polynomial of $A_i$, and $\xi_{i1},\dots,\xi_{i,w_i}$ its roots, in some order.

This data can best be understood in terms of a suitable root system. Associated to a quiver $Q$ with vertex set $I$ there is root system (consisting of real and imaginary roots, 
positive or negative) contained in the root lattice
\[
\Gamma_Q = \bigoplus_{v\in I} \Z \alpha_v = \{ \sum_{v\in I} n_v \alpha_v : n_v\in \Z \}.
\]
See for example \cite{Kacm}. The Euler form $\langle-,-\rangle_Q$ on $\Gamma_Q$ is given by
\[
\langle \sum_{v\in I} n_v \alpha_v, \sum_{v\in I} n'_v \alpha_v \rangle_Q = \sum_{v\in I} n_v n'_v - \sum_{a\in Q} n_{t(a)} n'_{h(a)},
\]
where $h(a)$ and $t(a)$ are the head and tail vertices of an arrow $a$ in $Q$. Its symmetrisation is $(\alpha,\beta)_Q=\langle\alpha,\beta\rangle_Q+\langle\beta,\alpha\rangle_Q$ and we define $p_Q(\alpha)=1-\langle\alpha,\alpha\rangle$.

Given a weight sequence $\bw$, let $Q_\bw$ be the star-shaped quiver
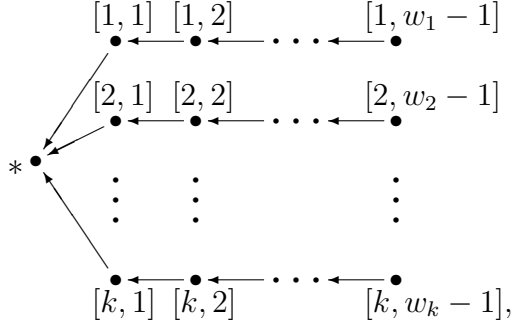
\setlength{\unitlength}{1.5pt}
\[
\begin{picture}(110,80)
\put(10,40){\circle*{2.5}} \put(30,10){\circle*{2.5}}
\put(30,50){\circle*{2.5}} \put(30,70){\circle*{2.5}}
\put(50,10){\circle*{2.5}} \put(50,50){\circle*{2.5}}
\put(50,70){\circle*{2.5}} \put(100,10){\circle*{2.5}}
\put(100,50){\circle*{2.5}} \put(100,70){\circle*{2.5}}
\put(28,67){\vector(-2,-3){16}}
\put(27,48.5){\vector(-2,-1){14}}
\put(28,13){\vector(-2,3){16}}
\put(47,10){\vector(-1,0){14}}
\put(47,50){\vector(-1,0){14}}
\put(47,70){\vector(-1,0){14}}
\put(67,10){\vector(-1,0){14}}
\put(67,50){\vector(-1,0){14}}
\put(67,70){\vector(-1,0){14}}
\put(97,10){\vector(-1,0){14}}
\put(97,50){\vector(-1,0){14}}
\put(97,70){\vector(-1,0){14}}
\put(70,10){\circle*{1}}
\put(75,10){\circle*{1}} \put(80,10){\circle*{1}}
\put(70,50){\circle*{1}} \put(75,50){\circle*{1}}
\put(80,50){\circle*{1}} \put(70,70){\circle*{1}}
\put(75,70){\circle*{1}} \put(80,70){\circle*{1}}
\put(30,25){\circle*{1}} \put(30,30){\circle*{1}}
\put(30,35){\circle*{1}} \put(50,25){\circle*{1}}
\put(50,30){\circle*{1}} \put(50,35){\circle*{1}}
\put(100,25){\circle*{1}} \put(100,30){\circle*{1}}
\put(100,35){\circle*{1}} \put(3,36.3){$*$} \put(24,2){$[k,1]$}
\put(24,54){$[2,1]$} \put(24,74){$[1,1]$} \put(44,2){$[k,2]$}
\put(44,54){$[2,2]$} \put(44,74){$[1,2]$} \put(92,2){$[k,w_k-1],$}
\put(92,54){$[2,w_2-1]$} \put(92,74){$[1,w_1-1]$}
\end{picture}
\]
with vertex set $I_\bw=\{ * \} \cup \{ [i,p] : 1\le i\le k, 1\le p < w_i \}$.
To simplify notation we write a subscript $[i,p]$ as $ip$ or $i,p$. We write $\Gamma_\bw$ for the root lattice $\Gamma_{Q_\bw}$, and $p_\bw$ for the function $p_{Q_\bw}$, so elements of $\Gamma_\bw$ are of the form
\[
\alpha = n_* \alpha_* + \sum_{i=1}^k \sum_{p=1}^{w_i-1} n_{ip} \alpha_{ip} \quad (n_*, n_{ip}\in\Z),
\]
and
\[
p_\bw(\alpha) = 1 - n_*^2 - \sum_{i=1}^k\sum_{p=1}^{w_i-1} n_{ip}^2
+ \sum_{i=1}^k\sum_{p=1}^{w_i-1} n_{ip-1} n_{ip}
\]
with the convention that $n_{i0}=n_\ast$ and $n_{iw_i}=0$. Given $\xi = (\xi_{ip})$ as above and $\alpha\in\Gamma_\bw$ we define
\[ \xi^{[\alpha]} = \prod_{i=1}^k \prod_{p=1}^{w_i} \xi_{ip}^{n_{ip-1} - n_{ip}} \]
again with $n_{i0} = n_*$ and $n_{iw_i}=0$.

In our parameterisation of the conjugacy classes $C_i$ the data $\bw,n_{ip}$ can thus be encoded as an element $\alpha_C=n\alpha_\ast+\sum_{i,p}n_{ip}\alpha_{ip}\in \Gamma_\bw$. Our main result is now as follows.

\begin{thm}
\label{t:dsp}
For there to be an irreducible solution to $A_1\dots A_k = 1$ with matrices $A_i\in C_i$
it is necessary and sufficient that $\alpha=\alpha_C$ be a positive root, that
$\xi^{[\alpha]}=1$, and $p_\bw(\alpha) > p_\bw(\beta)+p_\bw(\gamma)+\cdots$ for
any nontrivial decomposition of $\alpha$ as a sum of positive roots
$\alpha = \beta+\gamma+\cdots$ with $\xi^{[\beta]} = \xi^{[\gamma]} = \cdots = 1$.
\end{thm}

The sufficiency is proved in \cite{CBSh}. In section \ref{s:mainproof} we prove necessity (the combination of Lemma~\ref{l:dspmulcorr} and Theorem~\ref{t:mainthmformultpreproj}).

This result was already announced in \cite{CBicm}. The original plan was to prove the results of section~\ref{s:exttub} along the lines of \cite[Theorem 9.1]{CBmm}, with parabolic bundles and their connections replacing representations of quivers and deformed preprojective algebras. However, the general position argument used in \cite{CBmm} turned out to be difficult to write up, so this paper is based on the argument of~\cite{CBHu} instead.

A draft of this paper was completed already in 2018. Our plan was to combine the draft with results of the second author on coherent sheaves on weighted projective lines and connections on such sheaves, clarifying and improving the results in~\cite{CBcw}.
These results appeared in the second author's 2023 Habilitation thesis \cite{Hh}, and he announced a solution of the Deligne-Simpson Problem using them in his lectures at the 21st International Conference on Representations of Algebras in Shanghai in August 2024. But that work is still not completely written up, and with the announcement of another proof of the Deligne-Simpson Problem, obtained independently by Cheng Shu \cite{Shu}, we decided to revert back to the original, more direct, approach from 2018.

\section{Weighted projective lines and parabolic bundles}

Let $\XX$ be the weighted projective line over $\C$ consisting of the projective line $X=\PP^1$, 
a collection $D = (a_1,a_2,\dots,a_k)$ of distinct marked points in $\PP^1$,
and a weight sequence $\bw=(w_1,\dots,w_k)$ with $w_i\ge 1$
(where $w_i=1$ is equivalent to the point $a_i$ being unmarked).

Let $\Coh\XX$ be the category of coherent sheaves on $\XX$ defined by Geigle and Lenzing \cite{GL}. This is an hereditary abelian category; we call its objects \emph{parabolic sheaves}, and subobjects will be called \emph{parabolic subsheaves}. The category admits a split torsion pair, and the torsion parabolic sheaves form a uniserial length category. There is one simple torsion parabolic sheaf $S_a$ supported at each point $a\notin D$, and at the point $a_i$ there are simple torsion parabolic sheaves $S_{ip}$ ($0\le p < w_i$) with $\Ext^1(S_{ip},S_{iq})\neq0$ if and only if $p\equiv q+1$ modulo $w_i$.

There is a parabolic structure sheaf $\OO=\OO_\XX$, as well as a faithful action of the group $\mathbf L\coloneqq(\Z\oc\oplus\bigoplus_i\Z\ox_i)/\langle w_i\ox_i=\oc\;\forall i\rangle$. In particular, we have the standard short exact sequences
\[ 0\to \OO((p-1)\ox_i) \to \OO((p\ox_i) \to S_{ip} \to 0. \]
Moreover, Geigle and Lenzing show that the element $\oom = (k-2)\oc - \sum_{i=1}^k \ox_i \in \mathbf L$ induces Serre duality on $\Coh\XX$ in the form
\[ \Ext^1(\E,\F) \cong \Hom(\F,\E(\oom))^*. \]
In particular we have $S_{ip}(\oom)\cong S_{ip-1}$.

We can identify the Grothendieck group $K_0(\Coh\XX)$ with $\hat{\Gamma}=\Gamma\oplus\Z \partial$, where $\Gamma = \Gamma_\bw$ is the root lattice for the quiver $Q_\bw$, such that
\[ [\OO(k\oc)] = \alpha_* + k\partial, \quad [S_a] = \partial, \quad
[S_{ip}] = \begin{cases} \alpha_{ip} & (p\neq 0) \\
\partial - \sum_{q=1}^{w_i-1} \alpha_{iq} & (p=0). \end{cases} \]
We remark that we have followed \cite{Schiffmann}, rather than \cite{GL}, in our numbering of the $S_{ip}$ and the subsequent identification of the Grothendieck group.

The \emph{dimension vector} $\dimv \E\in \Gamma$ and degree $\deg_{\PP^1}\E \in\Z$ are defined by 
\[
[\E]=\dimv \E + (\deg_{\PP^1} \E)\partial
\]
and the \emph{rank} of $\E$ is the coefficient of $\alpha_\ast$ in $\dimv\E$.

The hereditary category $\Coh\XX$ has finite dimensional homomorphism and extension spaces, so its Euler form
\[ \langle \E,\F \rangle_\XX = \dim\Hom(\E,\F) - \dim\Ext^1(\E,\F) \]
descends to a bilinear form $\hat\Gamma\times\hat\Gamma\to\Z$ given as follows.

\begin{lem}
For $\alpha,\alpha'\in\Gamma$ we have
\[ \langle \alpha + k\partial, \alpha' + k'\partial \rangle_\XX =
\langle \alpha,\alpha'\rangle_{Q_\bw} + k' n_* - k n'_* \]
where $\alpha=n_\ast\alpha_\ast+\sum_{i,p}n_{ip}\alpha_{ip}$, and similarly for $\alpha'$.
\end{lem}

\begin{proof}
Straightforward.
\end{proof}

We remark that there is an exact functor $\pi\colon\Coh\XX\to\Coh\PP^1$, whose kernel is the Serre subcategory generated by $S_{ip}$ for $p\neq0$, and which preserves the rank and degree, so
\[ \rank\E = \rank(\pi\E) \quad\textrm{and}\quad \deg_{\PP^1}\E = \deg(\pi\E). \]
Moreover, $\pi$ admits fully faithful (and exact) left and right adjoints $\pi_!$ and $\pi_\ast$, so determines a recollement. We then have $\OO_\XX=\pi_!(\OO_{\PP^1})$ and $S_a=\pi_!(S_a)$ for all $a\notin D$.

Torsion-free parabolic sheaves on $\XX$ are also called \emph{parabolic bundles} on $\PP^1$, and following Lenzing \cite{L} they may be identified with collections $\E = (E,E_{ip})$ consisting of an (algebraic or holomorphic) vector bundle $E$ on $\PP^1$ and flags of vector subspaces
\[
E_{a_i} = E_{i0} \supseteq E_{i1} \supseteq \dots \supseteq E_{i,w_i-1} \supseteq E_{i,w_i} = 0
\]
of the fibres $E_{a_i}$ at each of the marked points $a_i$. In this case $E=\pi(\E)$, and our above description of the Grothendieck group was made in such a way that
\[
\dimv \E = (\rank E)\alpha_\ast + \sum_{i,p} (\dim E_{ip})\alpha_{ip} \in \Gamma.
\]
Morphisms between parabolic bundles are now given by those morphisms of vector bundles which preserve the flags. Up to isomorphism, the parabolic line bundles, equivalently the indecomposable rank one parabolic sheaves, are precisely the $\OO(\vec x)$ for $\vec x\in\mathbf L$.

We denote by $\Bun\XX$ the category of parabolic bundles on $\XX$.

\begin{lem}
\label{l:homtoshift}
Given parabolic bundles $\E = (E,E_{ip})$ and $\F = (F,F_{ip})$, and $1\leq j\leq k$, we can identify $\Hom(\F,\E(-\ox_j))$ with those homomorphisms $\theta\colon\F\to\E$ satisfying $\theta_{a_j}(F_{jq})\subseteq E_{jq+1}$ for all $0\leq q<w_j$.

Inductively, we can identify $\Hom(\F,\E(-\sum_i\ox_i))$ with those homomorphisms $\theta\colon\F\to\E$ satisfying $\theta_{a_i}(F_{ip})\subseteq E_{ip+1}$ for all $i$ and $p$.
\end{lem}

\begin{proof}
We claim that there is a monomorphism $\E(-\ox_j)\rightarrowtail\E$, so we can identify homomorphisms $\F\to\E(-\ox_j)$ with those homomorphisms $\theta\colon\F\to\E$ whose composite $\bar\theta\colon\F\to\E/\E(-\ox_j)$ vanishes.

On parabolic bundles the twist functor $\E\mapsto\E(-\ox_j)$ is given by `rotating cycles' \cite[\S4.1]{L}. Explicitly, given a parabolic bundle $\E=(E,E_{ip})$, we can form the pullback diagrams
\[ \begin{tikzcd}
0 \arrow[r] & E(-a_i) \arrow[r] \arrow[d,equal] & E^{ip} \arrow[r] \arrow[d,tail] & E_{ip} \arrow[r] \arrow[d,tail] & 0\\
0 \arrow[r] & E(-a_i) \arrow[r] & E \arrow[r] & E_{a_i} \arrow[r] & 0
\end{tikzcd} \]
whose bottom row is the standard short exact sequence, and hence obtain a cycle of vector bundles
\[ \begin{tikzcd}
E(-a_i)=E^{iw_i} \arrow[r,tail] & E^{iw_i-1} \arrow[r,tail] & \cdots \arrow[r,tail] &
E^{i1} \arrow[r,tail] & E
\end{tikzcd}\]

To describe $\E(-\ox_j)$, we replace $E$ by $E'\coloneqq E^{j1}$, and on the $j$-th arm we use the cycle
\[ \begin{tikzcd}
E'(-a_j) \arrow[r,tail] & E(-a_j) \arrow[r,tail] & \cdots \arrow[r,tail] & E^{j2} \arrow[r,tail] & E'
\end{tikzcd}\]
The quotients $E'_{jq}\coloneqq E^{jq+1}/E'(-a_j)$ then determine a flag of subspaces in the fibre $E'_{a_j}$.

On the $i$-th arm for $i\neq j$ we have the exact commutative diagram
\[ \begin{tikzcd}
0 \arrow[r] & E'(-a_i) \arrow[r] \arrow[d,tail] & E' \arrow[r] \arrow[d,tail] & E'_{a_i} \arrow[r] \arrow[d] & 0\\
0 \arrow[r] & E(-a_i) \arrow[r] & E \arrow[r] & E_{a_i} \arrow[r] & 0
\end{tikzcd}\]
The vertical morphisms on the left and in the middle are injective, and their cokernels are torsion sheaves supported at $a_j$. Thus the vertical map on the right, between torsion sheaves supported at $a_i$, must be an isomorphism. We can therefore identify the flag of subspaces in $E_{a_i}$ with a flag of subspaces in $E'_{a_i}$.

The parabolic bundle $\E(-\ox_j)$ is now given by $(E',E'_{ip})$, it is a parabolic subbundle of $\E=(E,E_{ip})$, and the quotient $\E/\E(-\ox_j)$ is a semisimple torsion parabolic sheaf supported at $a_j$. In fact, this quotient is the direct sum of $S_{jq}\otimes_\C(E_{jq}/E_{jq+1})$ for $0\leq q<w_j$.

Finally, given $\theta\colon\F\to\E$, the composite $\bar\theta\colon\F\to\E/\E(-\ox_j)$ is determined by the fibre map $\theta_{a_j}$ via
\[ \begin{tikzcd}
F_{jq} \arrow[r,"\theta_{a_j}"] & E_{jq} \arrow[r,two heads] & E_{jq}/E_{jq+1}
\end{tikzcd} \]
Thus $\bar\theta$ vanishes if and only if $\theta_{a_j}(F_{jq})\subseteq E_{jq+1}$ for all $q$.    
\end{proof}

If $\E = (E,E_{ij})$ is a parabolic bundle and $F$ is a vector bundle, one gets a parabolic bundle $\E\otimes F$ by endowing the vector bundle $E\otimes F$ with the parabolic structure given by the subspaces
\[
E_{ij}\otimes F_{a_i} \subseteq E_{a_i}\otimes F_{a_i} \cong (E\otimes F)_{a_i}.
\]
Lenzing's description of coherent sheaves on $\XX$ as $p$-cycles of sheaves \cite[\S4.1]{L} immediately gives the following.

\begin{lem}
\label{l:ctwist}
We have $\E(\oc)\cong \E\otimes \OO(1)$.
\end{lem}

\section{Connections on parabolic bundles}

We will write $\Omega = \Omega^1_{\PP^1} (\log D)$ for the sheaf of differential 1-forms on $\PP^1$ with logarithmic poles on $D$. It is a line bundle on $\PP^1$ of degree $k-2$ so isomorphic to $\OO(k-2)$. For each point $a_i\in D$, the map sending a logarithmic form to its residue at $a_i$ induces a canonical isomorphism $\Omega_{a_i}\cong\C$. Explicitly, we can take as basis element $\frac{dz}{z}$ for any local co-ordinate around $a_i$ with $z(a_i)=0$.

Given a vector bundle $E$ on $\PP^1$, a \emph{logarithmic connection} on $E$, singular over $D$, is a homomorphism of sheaves of vector spaces
\[ 
\nabla \colon E\to E\otimes \Omega 
\]
satisfying the Leibnitz identity
\[ 
\nabla(fs) = f\nabla(s) + s\otimes df 
\]
for a local section $s$ of $E$ and a regular function $f$. Making the above identification of fibres, each logarithmic connection $\nabla$ on $E$ induces a linear endomorphism $\Res_{a_i}\nabla$ of $E_{a_i}$, called the \emph{residue} at $a_i$.

Extending ideas of Atiyah, Mihai constructed a functorial exact sequence
\[ \begin{tikzcd}
0 \arrow[r] & E\otimes\Omega \arrow[r] & M(E) \arrow[r] & E \arrow[r] & 0
\end{tikzcd}\]
such that logarithmic connections on a vector bundle $E$ are in bijection with sections of this sequence. He further described a necessary and sufficient condition for the existence of a logarithmic connection on a vector bundle in terms of the Euler form, and as a consequence obtained the following criterion. Here, as mentioned in \cite[\S7]{CBipb}, we use the opposite sign convention for residues.

\begin{lem}[{\cite[Corollaire 3]{Mihai}}]
\label{l:conntrace}
If $\nabla\colon E\to E\otimes \Omega$ is a logarithmic connection on a vector bundle $E$, then
\[ \sum_{i=1}^k \tr(\Res_{a_i} \nabla) = -\deg E. \]
\end{lem}

We remark that this sequence is a pushout of the sequence constructed by Atiyah, which corresponds to the case $D=\emptyset$. Also, as soon as $D$ contains at least two points, there will be torsion sheaves for which the above sequence splits.

Now let $\XX$ be a weighted projective line having weight sequence $\bw = (w_1,\dots,w_k)$, and fix a collection of complex numbers $\zeta = (\zeta_{ip})$ ($1\le i\le k$, $1\le p\le w_i$). In \cite{CBipb} the first author introduced the notion of a \emph{$\zeta$-connection} on a parabolic bundle $\E=(E,E_{ip})$, which is a logarithmic connection $\nabla$ on $E$ satisfying
\[ (\Res_{a_i}\nabla-\zeta_{ip})(E_{ip-1}) \subseteq E_{ip} \quad\textrm{for all }i,p, \]
and later in \cite{CBcw} constructed a functorial short exact sequence
\[ \begin{tikzcd}
0 \arrow[r] & \E(\oom)\arrow[r] & D_\zeta(\E)\arrow[r] & \E \arrow[r] & 0
\end{tikzcd} \]
with the property that $\zeta$-connections on a parabolic bundle $\E$ are in bijection with sections of this sequence. We again note that in general there will be torsion parabolic sheaves for which the above sequence splits.

We denote by $\CohConn_\zeta \XX$ the category of pairs $(\E,s)$ consisting of a parabolic sheaf $\E$ on $\XX$ together with a section $s:\E\to D_\zeta(\E)$ for the above sequence. The morphisms $\theta\colon(\E,s)\to(\E',s')$ are morphisms of parabolic sheaves $\theta\colon\E\to\E'$ which commute with the sections, $s'\theta=D_\zeta(\theta)s$. It is shown in \cite{Hh} that $\CohConn_\zeta\XX$ is a length category.

The subobjects of $(\E,s)$ correspond to parabolic subsheaves $i\colon\F\rightarrowtail\E$ such that $si$ factors through $D_\zeta(i)$; equivalently the following composite morphism vanishes
\[ \begin{tikzcd}
\F \arrow[r,"i"] & \E \arrow[r,"s"] & D_\zeta(\E) \arrow[r] & D_\zeta(\E/\F).
\end{tikzcd}\]
In this case we also call $\F$ an \emph{invariant parabolic subsheaf} of $(\E,s)$. The \emph{simple}, or \emph{irreducible}, objects in $\CohConn_\zeta\XX$ are those pairs $(\E,s)$ admitting no proper non-zero invariant parabolic subsheaves.

We record the following easy criterion for the existence of an invariant parabolic subsheaf.

\begin{lem}
\label{l:dirsumnotirred}
Suppose $\E=\E^1\oplus\E^2$ with $\Ext^1(\E^2,\E^1)=0$. Then $\E^1$ is an invariant parabolic subsheaf of any $(\E,s)\in\CohConn_\zeta\XX$.
\end{lem}

\begin{proof}
We observe that, by functoriality, the composite
\[ \begin{tikzcd}
\E^1 \arrow[r] & \E \arrow[r,"s"] & D_\zeta(\E) \arrow[r] & D_\zeta(\E^2) \arrow[r] & \E^2
\end{tikzcd}\]
equals the composite $\E^1\to\E\to\E^2$, so vanishes. Thus the morphism $\E^1\to D_\zeta(\E^2)$ factors through $\E^2(\oom)$, but $\Hom(\E^1,\E^2(\oom))\cong\Ext^1(\E^2,\E^1)^\ast=0$ by Serre duality.
\end{proof}

We write $\BunConn_\zeta\XX$ for the full subcategory of pairs $(\E,s)$, equivalently pairs $(\E,\nabla)$, where $\E=(E,E_{ip})$ is a parabolic bundle. A parabolic subsheaf $\F=(\F,F_{ip})$ of $\E$ is invariant for $\nabla$ if and only if
\[ (\Res_{a_i}\nabla-\zeta_{ip})(F_{ip-1})\subseteq F_{ip} \quad\textrm{for all }i,p. \]
In particular, given $F\subseteq E$, we can obtain one invariant parabolic subsheaf by setting $F_{ip}\coloneqq F_{a_i}\cap E_{ip}$ for all $i,p$.

Given $\alpha = n_* \alpha_* + \sum_{i=1}^k \sum_{p=1}^{w_i-1} n_{ip} \alpha_{ip}\in\Gamma_\bw$ we define
\[ \zeta*{[\alpha]} \coloneqq \sum_{i=1}^k \sum_{p=1}^{w_i} \zeta_{ip}{(n_{ip-1} - n_{ip})}, \]
with the convention that $n_{i0} = n_*$ and $n_{i,w_i}=0$. It now follows from Lemma~\ref{l:conntrace} that if $\nabla$ is a $\zeta$-connection on a parabolic bundle $\E$, then $\deg_{\PP^1}\E+\zeta\ast[\dimv\E]=0$. The next result shows that this essentially characterises when $\zeta$-connections exist.

\begin{lem}
\label{l:degdimconn}
A parabolic bundle $\E$ admits a $\zeta$-connection if and only if $\deg_{\PP^1} \E' +\zeta*[\dimv \E']=0$ for all indecomposable direct summands $\E'$ of $\E$. In particular, if there exists $(\E,\nabla)\in\BunConn_\zeta\XX$, then $\deg_{\PP^1} \E +\zeta*[\dimv \E]=0$.
\end{lem}

\begin{proof}
This is \cite[Theorem 7.1]{CBipb}. In \cite{Hh} this is generalised to cover all parabolic sheaves.
\end{proof}

\section{The tubular case}
\label{s:tub}

We begin with some standard definitions which are valid for all $\XX$, but particularly useful in the tubular case. Let $\XX$ be a weighted projective line having weight sequence $\bw = (w_1,\dots,w_k)$. Let $w$ be the least common multiple of the components of $\bw$. We write $\deg_{\XX} \E$ for the degree of a parabolic sheaf in the sense of Geigle and Lenzing \cite[Proposition 2.8]{GL},
\[
\deg_{\XX} \E = w \deg_{\PP^1} \E + \sum_{i=1}^k \sum_{p=1}^{w_i-1} n_{ip} w/w_i \in \Z
\]
where  the $n_{ip}$ are given by
\[
\dimv \E = n_* \alpha_* + \sum_{i=1}^k \sum_{p=1}^{w_i-1} n_{ip} \alpha_{ip}.
\]
The \emph{slope} of a parabolic sheaf $\E$ is $\mu(\E) = \deg_{\XX} \E / \rank \E$, see \cite[\S 2.5]{LM},
so parabolic bundles have rational slope, and the torsion parabolic sheaves have slope $\infty$.
A parabolic sheaf $\E$ is \emph{semistable} (respectively \emph{stable}) if for every proper non-zero parabolic subsheaf $\F$ of $\E$ we have $\mu(\F)\le \mu(\E)$ (respectively $\mu(\F)<\mu(\E))$.

In the rest of this section we assume that $\XX$ is of \emph{tubular} type, or of \emph{virtual genus} 1, 
meaning that $\sum_{i=1}^k 1/w_i = k-2$. Assuming that the $w_i$ are non-increasing, the possible weight types are $(2,2,2,2)$, $(3,3,3)$, $(4,4,2)$ and $(6,3,2)$, and thus $w=w_1$. Observe that the quiver $Q_\bw$ is extended Dynkin, with $0=[1,w_1-1]$ as an extending vertex. Let $\delta$ be the minimal positive imaginary root for $Q_\bw$, which we recall satisfies $\delta_{ip-1}=\delta_{ip}+w/w_i$ for $1\leq p\leq w_i$. In particular, the coefficient of $\alpha_*$ is $w$. We note also that $\delta$ is a radical vector, so $(\delta,-)_{Q_\bw}=0$, and conversely every radical vector is an integer multiple of $\delta$.

\begin{lem}
\label{l:degxprop}
(i) For a parabolic sheaf $\E$ we have 
\[ \deg_{\XX} \E = \langle \delta - w\partial , [\E]\rangle_\XX
= w \deg_{\PP^1} \E + w \rank \E + \langle \delta , \dimv \E \rangle_{Q_\bw}. \]
In particular, if $\dimv\E \in \Z \delta$, then $\deg_\XX \E \in \Z w$.

(ii) For parabolic bundles $\E$ and $\E'$ of ranks $r$ and $r'$ we have
\[ w^2\langle [\E],[\E']\rangle_\XX = w r r' (\mu(\E')-\mu(\E)) +
\langle w \dimv \E - r \delta, w \dimv \E' - r' \delta \rangle_{Q_\bw}. \]
\end{lem}

\begin{proof}
Straightforward calculation.
\end{proof}

The representation theory of a tubular weighted projective line has been worked out by Geigle and Lenzing \cite{GL} and Lenzing and Meltzer \cite{LM}. The indecomposable parabolic bundles of a given slope $q\in\Q$ are all semistable, from which it follows that $\Hom(\E,\F)=0$ for indecomposables $\E$ and $\F$ with $\mu(\E)>\mu(\F)$. Next, the indecomposables of slope $q$ determine a uniserial length category $\CC_q$, stable under Serre duality. The Auslander--Reiten components are all tubes \cite[Theorem 5.6]{GL}, and the relative simple objects in $\CC_q$ are precisely the stable parabolic sheaves. Finally, the category $\CC_q$ is equivalent to the category of torsion parabolic sheaves \cite[Theorem 4.4]{LM}, so the ranks of the inhomogeneous tubes are equal to the components of $\bw$.

\begin{lem}
\label{l:slopeqprop}
Let $T$ be a tube of slope $q\in\Q$ and set $\E$ to be the direct sum of one copy of each stable parabolic bundle in $T$.

(i) We have $\dimv \E = a\delta$, where $a=\rank\E/w$ is a positive integer.

(ii) The dimension vectors of the stable parabolic bundles in $T$ are linearly independent.

(iii) Each indecomposable parabolic bundle of slope $<q$ has a non-zero map to $\E$
and each indecomposable parabolic bundle of slope $>q$ has a non-zero map from $\E$.
\end{lem}

\begin{proof}
(i) We know that $\E \cong \E(\oom)$, so 
\[
\langle \F,\E\rangle_\XX = \langle \F,\E(\oom)\rangle_\XX = - \langle \E,\F\rangle_\XX.
\]
Using Lemma~\ref{l:slopeqprop}, together with the fact that $\delta\in\Gamma$ is radical, we get
\[ (\dimv\E,\dimv\F)_{Q_\bw} = \langle\E,\F\rangle_\XX + \langle\F,\E\rangle_\XX = 0 \]
for all $\F$, and hence that $\dimv\E=a\delta$ for some integer $a$. Now note that $\rank\E=aw$, so $a$ is positive.

(ii) Let $T$ have rank $r$, and stable parabolic bundles $\mathcal S_i$ for $i\in\Z/r\Z$, indexed so that $\mathcal S_{i-1}\cong\mathcal S_i(\oom)$. Then
\[ \langle\mathcal S_i,\mathcal S_j\rangle_\XX = \mathbf 1_{[j=i]} - \mathbf 1_{[j=i-1]}. \]
Suppose some linear combination of their dimension vectors vanishes, say $\sum_i m_i \dimv \mathcal{S}_i = 0$. Then $\sum_i m_i\rank\mathcal S_i=0$, as this is the coefficient of $\alpha_\ast$, so
\[
\sum_i m_i \left(  w \dimv \mathcal{S}_i - (\rank \mathcal{S}_i) \delta \right) = 0.
\]
Thus for fixed $j$, Lemma~\ref{l:slopeqprop} gives
\[ 0 = \sum_i m_i \langle\mathcal S_i,\mathcal S_j\rangle_\XX = m_j - m_{j+1}. \]
We deduce that the $m_i$ are all equal, and then necessarily zero as $\sum_i\dimv\mathcal S_i>0$ by (i).

(iii) Take $\F\in\CC_{q'}$, say of rank $r'$. Then by (i), together with Lemma~\ref{l:slopeqprop}, we have
\[ \langle\E,\F\rangle_\XX = ar'(q'-q) = -\langle\F,\E\rangle_\XX. \]
If $q'>q$, then $\langle\E,\F\rangle_\XX>0$ and so there exists a non-zero homomorphism $\E\to\F$. Similarly, if $q'<q$, then there exists a non-zero homomorphism $\F\to\E$.
\end{proof}

For later use we record the following.

\begin{lem}
\label{l:dimsdistincttubes}
Suppose $\E\oplus\F\in\CC_q$ has dimension vector a multiple of $\delta$, where $\E$ and $\F$ are supported on distinct sets of tubes. Then $\dimv\E$ and $\dimv\F$ are also both multiples of $\delta$. In particular, for a tube $T\subset\CC_q$, each stable $\mathcal S\in T$ occurs with the same multiplicity as a composition factor of $\E$.
\end{lem}

\begin{proof}
We have
\[ \langle\E,\E\rangle_\XX = \langle\E,\E\oplus\F\rangle_\XX - \langle\E,\F\rangle_\XX = 0 \]
where the first summand in the middle vanishes by Lemma~\ref{l:slopeqprop}, and the second summand vanishes since $\E$ and $\F$ have disjoint supports. Thus $w\dimv\E-(\rank\E)\delta$ is a multiple of $\delta$, and hence so too is $\dimv\E$. The second statement now follows from Lemma~\ref{l:slopeqprop}.  
\end{proof}

We define the (right) perpendicular category of a parabolic sheaf $\F$ to be
\[ \F^\perp \coloneqq \{\E\in\Coh\XX \mid \Hom(\F,\E) = 0 = \Ext^1(\F,\E)\}. \]
This is an hereditary abelian subcategory, closed under extensions.

\begin{lem}
\label{l:HubnerLenzing}
If $\F$ is a stable parabolic bundle without self-extensions, then $\F^\perp$ is equivalent to the category of finite dimensional representations of an extended Dynkin quiver.
\end{lem}

\begin{proof}
By H\"ubner and Lenzing \cite{HL} (see also \cite[Theorem 4.2]{CBcw}), the perpendicular category is equivalent to the category of finite dimensional $A$-modules for some finite dimensional hereditary algebra $A$. It is therefore enough to check that $A$ is tame and connected.

Let $\F$ have slope $q$. We know that $\CC_q$ is a tubular family, and the tubes not containing $\F$ all belong to $\F^\perp$. Also, the intersection of the tube containing $\F$ with $\F^\perp$ is again a tube (of rank one less). We see that the Auslander--Reiten quiver of $A$ contains tubes, so $A$ must have a tame hereditary factor. Also, every indecomposable $A$-module is either in a tube, or has a non-zero map to or from a fixed tube, so $A$ is connected.
\end{proof}

\begin{lem}
\label{l:slopeqclass}
Let $T$ be a tube of slope $q\in\Q$, and $\E$ the direct sum of the stable parabolic bundles in $T$. Then $[\E]=a\delta+d\partial$ for coprime integers $a,d$ satisfying $a>0$ and $d/a=q-w$.
\end{lem}

\begin{proof}
Choose a different inhomogeneous tube $T'$ of slope $q$, and a stable parabolic bundle $\F\in T'$. By the previous lemma we have an equivalence $\F^\perp\cong\mmod A$ for some connected tame hereditary algebra $A$. As $T\subset\F^\perp$, we can regard it as a tube in $\mmod A$, and it is well-known that $[\E]$ equals the minimal positive imaginary positive root $\delta_A\in K_0(A)$.

Using Lemma~\ref{l:slopeqprop} (i) we have $[\E]=a\delta+d\partial$ for integers $a,d$ with $a$ positive, and as $\deg_\XX\E=w\deg_{\PP^1}\E+w\rank\E=dw+aw^2$ we see that $q=\mu(\E)=d/a+w$. Finally, we have $K_0(\Coh\XX)=K_0(\F^\perp)\oplus\Z[\F]$ and also $K_0(\F^\perp)=K_0(A)$. As $\delta_A\in K_0(A)$ is indivisible, so too is $a\delta+d\partial\in K_0(\Coh\XX)$, showing that $a$ and $d$ are coprime.
\end{proof}

\begin{prop} 
\label{p:notubularirred}
Suppose $(\E,\nabla)\in\BunConn_\zeta\XX$ has $\dimv\E=hm\delta$ where $s=\zeta\ast[m\delta]\in\Z$. If $h\geq2$, then $(E,\nabla)$ is not irreducible.
\end{prop}

\begin{proof}
Suppose first that $\E$ has indecomposable summands from at least two different tubes. In this case there is a non-trivial decomposition $\E = \E^1\oplus\E^2$ with $\Ext^1(\E^2,\E^1)=0$, so Lemma~\ref{l:dirsumnotirred} applies.

We are reduced to the case when $\E$ is contained in a single tube $T$, say of slope $q$. By Lemma~\ref{l:degdimconn} we have $\deg_{\PP^1}\E=-\zeta\ast[hm\delta]=-hs$, so $[\E]=h(m\delta-s\partial)$ and $q=w-s/m$. Moreover, we may further assume that $m>0$ is minimal such that $\zeta\ast[m\delta]\in\Z$, and hence that $m$ and $s$ are coprime.

Choose a different inhomogeneous tube $T'$ of slope $q$, and a stable parabolic bundle $\F\in T'$. By Lemma~\ref{l:HubnerLenzing} the category $\F^\perp$ is equivalent to $\mmod\C Q$ for some acyclic extended Dynkin quiver $Q$, the subcategory $\C_q\cap\F^\perp$ is equivalent to the subcategory of regular modules for $\C Q$, and $m\delta-s\partial$ corresponds to the minimal imaginary positive root $\delta_Q$. Thus $\E$ corresponds to a regular module $M$ of dimension vector $h\delta_Q$.

Finally, by \cite{CBcw}, the category of objects $(\G,s)\in\CohConn_\zeta\XX$ with $\G\in\F^\perp$ is equivalent to finite dimensional modules over some deformed preprojective algebra $\Pi^\lambda Q$. Thus $(\E,\nabla)$ corresponds to some $\Pi^\lambda Q$-module $\widetilde M$, restricting to the $\C Q$-module $M$. As $[M]=h\delta_Q$, necessarily $\lambda\cdot hb\delta_Q=0$, and hence also $\lambda\cdot\delta_Q=0$.

We are thus in the situation of \cite[Theorem 1.2]{CBmm}, which tells us that if $h\geq2$, then $\widetilde M$ is not simple, and hence $(\E,\nabla)$ is not irreducible. In fact, all that is needed is the argument of Case (I) in \cite[\S10]{CBmm}; alternatively one could use Crawley-Boevey and Hubery \cite{CBHu}.
\end{proof}

\section{The extended tubular case}
\label{s:exttub}

Let $\XX$ be a weighted projective line having weight sequence $\bw = (w_1,\dots,w_k)$. In this section we assume that $\XX$ is of \emph{extended tubular} type, meaning that $\bw$ is one of $(3,2,2,2)$, $(4,3,3)$, $(5,4,2)$, or $(7,3,2)$. Thus the quiver $Q_\bw$ is over-extended Dynkin: starting from an extended Dynkin diagram with extending vertex $0$, we adjoin a further vertex $\infty$ and a single arrow $\infty\to0$. In our case we have $0=[1,w_1-2]$ and $\infty=[1,w_1-1]$.

We write $S_\infty$ and $S_0$ for the simple torsion parabolic sheaves corresponding to the vertices $\infty$ and $0$, and observe that $S_0\cong S_\infty(\oom)$, and $S_\infty=S_{1,0}(\oom)$.

Let $\XX'$ be the weighted projective line of tubular type having weight sequence $\bw'=(w_1-1,w_2,\ldots,w_k)$ and the same marked points as $\XX$. As observed above, $0=[1,w_1-2]$ is an extending vertex for $Q_{\bw'}$. Also, $w'=w_1-1$ is the least common multiple of the components of $\bw'$. We denote by $\delta$ the minimal positive imaginary root for $Q_{\bw'}$, and write $\mu$ for the slope function for $\XX'$.

There is an exact functor $\pi'\colon\Coh\XX\to\Coh\XX'$, called \emph{reduction of weight} in \cite{GLperp}, and inducing a recollement between the two categories. The kernel of $\pi'$ is generated by $S_\infty$, and the fully faithful left and right adjoints $\pi'_!$ and $\pi'_\ast$ of $\pi$ are in fact both exact. The following lemma records some of their properties.

\begin{lem}
\label{l:recollement}
(i) On parabolic bundles, $\pi'(E,E_{ip})$ forgets the vector space $E_\infty$.

(ii) The left adjoint $\pi'_!$ induces an equivalence $\Coh\XX'\cong {}^\perp S_\infty$. On parabolic bundles, $\pi'_!(E,E_{ip})$ extends the parabolic structure by setting $E_\infty=0$. We have $\Ext^1_\XX(\pi'_!\F,\E) \cong \Ext^1_{\XX'}(\F,\pi'\E)$ for $\E\in\Coh\XX$ and $\F\in\Coh\XX'$.

(iii) The right adjoint $\pi'_\ast$ induces an equivalence $\Coh\XX'\cong S_\infty^\perp$. On parabolic bundles, $\pi'_\ast(E,E_{ip})$ extends the parabolic structure by setting $E_\infty=E_0$. We have $\Ext^1_\XX(\E,\pi'_\ast\F) \cong \Ext^1_{\XX'}(\pi'\E,\F)$ for $\E\in\Coh\XX$ and $\F\in\Coh\XX'$.
\end{lem}

\begin{proof}
For the second part of (ii) we note that $\dim\Hom_\XX(\E,S_\infty)=\langle\E,S_\infty\rangle_\XX=n_\infty$. Also, the kernel and cokernel of the counit $\pi'_!\pi'\E\to\E$ both lie in $\Ker\pi'=\add S_\infty$, so the third part follows by applying $\Hom_\XX(\pi'\F,-)$.
\end{proof}

We have $K_0(\Coh\XX)=K_0({}^\perp S_\infty)\oplus\Z[S_\infty]$, and we can use the left adjoint $\pi'_!$ to identify ${}^\perp S_\infty\cong\Coh\XX'$. This determines an inclusion $\Gamma_{\bw'}\subset\Gamma_\bw$ with image those elements not supported at $\alpha_\infty$. In particular this allows us to regard $\delta$ as an element of $\Gamma_\bw$.

Our main theorem in this section is the following.

\begin{thm}
\label{t:exttubular}
Let $\E=(E,E_{ip})$ be a parabolic bundle on $\XX$ with $\dimv\E=\alpha_\infty+h\delta$. Assume further that $s=\zeta\ast[\delta]\in\Z$ and $\zeta\ast[\alpha_\infty]=0$. If $h\geq2$, then any $(\E,\nabla)\in\BunConn_\zeta\XX$ is not irreducible.
\end{thm}

We will divide the proof into three propositions, according to how $\pi'\E$ decomposes in $\Coh\XX'$. For the rest of this section we fix a parabolic bundle $\E=(E,E_{ip})$ on $\XX$ having dimension vector $\dimv\E=\alpha_\infty+h\delta$, as well as a tuple of complex numbers $\zeta=(\zeta_{ip})$ such that $\zeta\ast[\alpha_\infty]=0$ and $s\coloneqq\zeta\ast[\delta]\in\Z$. We also set $\F=\pi'\E$, which is the parabolic bundle on $\XX'$ given by $\F=(E,E_{ip})$, but forgetting the vector space $E_\infty$.

First note that, by Lemma~\ref{l:degdimconn}, $\deg_{\PP^1}\E=-\zeta\ast[\dimv\E]=-hs$. As this is also the degree of $\F$ we get $[\F]=h(\delta-s\partial)$, and so $q=\mu(\F)=w'-s\in\Z$ by Lemma~\ref{l:degxprop}.

Suppose $\F=\F^1\oplus\F^2$ in $\Coh\XX'$. Then \textit{a fortiori} $E=F^1\oplus F^2$ in $\Coh\PP^1$, so the logarithmic connection $\nabla\colon E\to E\otimes\Omega$ decomposes as $(\nabla^{uv})$ for some $\C$-linear morphisms $\nabla^{uv}\colon F^v\to F^u\otimes\Omega$. The Leibniz identity tells us that $\nabla^{uu}$ is a logarithmic connection on $F^u$, whereas $\nabla^{uv}$ for $u\neq v$ is a morphism of sheaves. Similarly the residue $\Res_{a_i}\nabla$ decomposes as $(R_i^{uv})$, where $R_i^{uu}=\Res_{a_i}\nabla^{uu}$ is the residue of $\nabla^{uu}$, whereas $R_i^{uv}$ for $u\neq v$ is the composition
\[
F^v_{a_i}\to F^u_{a_i} \otimes \Omega_{a_i} \to F^u_{a_i}
\]
where the first map is the fibre of the morphism $\nabla^{uv}$ at $a_i$ and the second map is given by the canonical isomorphism $\Omega_{a_i} \cong \C$ mentioned at the start of \S3.

Finally, the condition $(\Res_{a_i}\nabla-\zeta_{ip})(E_{ip-1})\subseteq E_{ip}$ yields the conditions $(\Res_{a_i}\nabla^{uu}-\zeta_{ip})(F^u_{ip-1})\subseteq F^u_{ip}$ and $R_i^{uv}(F^v_{ip-1})\subseteq F_{ip}^u$ for all pairs $[i,p]$ other than $\infty=[1,w_1-1]$ and $[1,w_1]$.

We now set $\phi=(\phi^{uv})$ to be the restriction of $(\Res_{a_1}\nabla-\zeta_\infty)$ to $E_0=F^1_0\oplus F^2_0$ and observe that, since $\Ima(\phi)\subseteq E_\infty$, this matrix has rank one.

Our basic strategy will be to deduce that $\phi^{uv}=0$ for some pair $u\neq v$. Since $\oom = (k-2)\oc - \sum_{i=1}^k \ox_i$, and $\Omega\cong\OO(k-2)$, Lemmas~\ref{l:homtoshift} and \ref{l:ctwist} will then tell us that $\nabla^{uv}$ is a morphism of parabolic sheaves $\F^v\to\F^u(\oom_{\XX'})$ in $\Coh\XX'$. If we know that the only such homomorphism is zero, then we get an invariant parabolic subsheaf $\E^v$ of $(\E,\nabla)$ by taking $\F^v$ and extending it by $E^v_\infty\coloneqq\Ima\phi^{vv}$.

\begin{prop}
\label{p:exttubular1}
Suppose $\pi'\E\in\Coh\XX'$ has direct summands of distinct slopes. Then $(\E,\nabla)$ is not irreducible. In fact, it contains an invariant parabolic subsheaf restricting to a direct summand of $\pi'\E$.
\end{prop}

\begin{proof}
Set $\F\coloneqq\pi'\E$. Then $\F$ has slope $q=w'-s\in\Z$, so we can decompose $\F=\F^1\oplus\F^2\oplus\F^3$ such that $\F^1\in\CC_{<q}$, $\F^2\in\CC_q$, and $\F^3\in\CC_{>q}$. By assumption one, and hence both, of $\F^{1,3}$ is non-zero. Following our proof strategy we can decompose $\nabla=(\nabla^{uv})$ for $1\leq u,v\leq 3$, and similarly $\phi=(\phi^{uv})$.

Consider the parabolic subsheaf $\G$ of $\F^3$ where we replace $F^3_0$ by $G_0\coloneqq\Ker\phi^{13}$. Then $\nabla^{13}$ determines a sheaf homomorphism $\G\to\F^1(\oom_{\XX'})$. We claim, however, that $\G\in\CC_{\geq q}$, so the only such homomorphism is zero, and hence $\nabla^{13}=0$. To prove the claim let $\bar\G$ be a direct summand of $\G$ and consider the pushout diagram in $\Coh\XX'$
\[ \begin{tikzcd}
0 \arrow[r] & \G \arrow[r] \arrow[d,two heads] & \F^3 \arrow[r] \arrow[d,two heads] & S'_0 \arrow[r] \arrow[d,equal] & 0\\
0 \arrow[r] & \bar\G \arrow[r] & \bar\F^3 \arrow[r] & S'_0 \arrow[r] & 0
\end{tikzcd} \]
where $S'_0=\pi'(S_0)$ is the simple torsion parabolic sheaf in $\Coh\XX'$ corresponding to the extending vertex $0=[1,w_1-2]$.

The epimorphism $\F^3\twoheadrightarrow\bar\F^3$ shows that $\bar\F^3\in\CC_{>q}$, whereas $\deg_{\XX'}S'_0=\langle\delta,\alpha_0\rangle_{Q_{\bw'}}=1$ by Lemma~\ref{l:degxprop}. Setting $r=\rank\bar G=\rank\bar F^3$ we get
\[ \deg_{\XX'}\bar G = \deg_{\XX'}\bar F^3-1 > qr-1. \]
As both the degree and $qr$ are integers it follows that $\deg_{\XX'}\bar G\geq qr$ as required.

We have therefore shown that $\nabla^{13}=0$, and hence that $\phi^{13}=0$. As $\phi$ has rank at most one, either $\phi^{12}$ or $\phi^{23}$ must also vanish. In the first case we obtain an invariant parabolic subsheaf of $(\E,\nabla)$ restricting to $\F^2\oplus\F^3$, whereas in the second case there is an invariant parablic subsheaf restricting to $\F^3$.
\end{proof}

From now on we assume that $\F=\pi'\E$ lies in the subcategory $\CC_q$. The next case, when $\F$ lies in a single tube, is the most technical. We begin with the following result, generalising Lemma~\ref{l:HubnerLenzing}.

\begin{lem}
\label{l:HubnerLenzing2}
Let $T'\subset\Coh\XX'$ be an inhomogeneous tube of slope $q=w'-s\in\Z$. Then there exists a unique stable parabolic sheaf $\G'\in T'$ with $\Hom_{\XX'}(\G',\pi'S_{1,0})\neq0$. We set $\G\coloneqq\pi'_!\G'$. Then ${\G'}^\perp\cong\mmod\C Q'$ and $\G^\perp\cong\mmod\C Q$, with $Q$ an over-extended Dynkin quiver, arising from the extended Dynkin quiver $Q'$.

The vertex $\infty\in Q$ is a source, corresponding to the relative simple injective $S_\infty\in\G^\perp$, and the vertex $0\in Q'$ is also a source, corresponding to the relative simple injective $\pi'S_0\in{\G'}^\perp$.
\end{lem}

\begin{proof}
We know that $S'_{1p}\coloneqq\pi'S_{1p}$ for $[1,p]\neq\infty$ are the simple torsion parabolic sheaves in $\Coh\XX'$ supported at $a_1$, so under Serre duality they form a single orbit of size $w_1-1=w'$, and their direct sum has class $\partial\in K_0(\Coh\XX')$.

If now $\mathcal T'$ is the direct sum of the stable parabolic sheaves in $T'$, then this has class $\delta-s\partial$ by Lemma~\ref{l:slopeqclass}, which corresponds to the minimal imaginary positive root for $Q'$. As $\mathcal T'$ is stable under the Serre functor we see that
\[ w'\langle\mathcal T',S'_{1p}\rangle_{\XX'} = \langle\delta-s\partial,\partial\rangle_{\XX'} = w' \]
for any $[1,p]\neq\infty$.

It follows that there is a unique indecomposable direct summand of $\mathcal T'$ admitting a homomorphism to each $S'_{1p}$ for $[1,p]\neq\infty$. For $S'_{1,0}$ this gives a stable parabolic sheaf $\G'$. Applying the Serre functor shows that there is a non-zero homomorphism from $\G'(\oom_{\XX'})$ to $S'_{1,0}(\oom_{\XX'})=S'_0$, so by uniqueness we deduce that $S'_0$ lies in ${\G'}^\perp$.

It is clear that $S'_0$ is simple in ${\G'}^\perp$. To see that it is a relative injective, suppose $\Ext^1_{\XX'}(\mathcal H',S'_0)\neq0$ for some indecomposable parabolic sheaf $\mathcal H'\in\Coh\XX'$. Then $\mathcal H'$ is necessarily torsion with simple socle $S'_{1,0}$, which implies $\Hom_{\XX'}(\G',\mathcal H')\neq0$, and so $\mathcal H'\notin{\G'}^\perp$. Thus $S'_0$ is a relative injective, so the corresponding vertex $0\in Q'$ is a source. Moreover, as $\langle\delta-s\partial,[S'_0]\rangle_{\XX'}=1$, we see that $0$ is even an extending vertex.

Now set $\G\coloneqq\pi'_!\G'$, an indecomposable parabolic bundle in ${}^\perp S_\infty\subset\Coh\XX$ without self-extensions. Using Lemma~\ref{l:recollement} the adjoint functors $(\pi'_!,\pi')$ restrict to the categories $\G^\perp$ and ${\G'}^\perp$, and in fact we have a colocalisation sequence of hereditary abelian categories
\[ \begin{tikzcd}
\add S_\infty \arrow[r,yshift=-3pt] & \G^\perp \arrow[r,yshift=-3pt,"\pi'" swap] \arrow[l,yshift=3pt]
& {\G'}^\perp \arrow[l,yshift=3pt,"\pi'_!" swap]
\end{tikzcd} \]
In particular, $S_\infty$ is relative simple injective in $\G^\perp$.

By \cite{HL} we know that $\G^\perp$ is equivalent to $\mmod\C Q$ for some quiver $Q$, so $S_\infty$ corresponds to a source vertex $\infty$ in $Q$. The colocalisation sequence then tells us that $Q'$ is obtained from $Q$ by removing this vertex.

Finally, if $\Ext^1_\XX(S_\infty,\mathcal H)$ for some indecomposable $\mathcal H\in\Coh\XX$, then Serre duality gives $\Hom_\XX(\mathcal H,S_0)\neq0$, so $S_0$ is the only simple in $\G^\perp$ which can extend $S_\infty$. We already know that $\Ext^1_\XX(S_\infty,S_0)$ is one dimensional, so there is a unique arrow in $Q$ starting at $\infty$, and this ends at $0$.
\end{proof}

\begin{prop}
\label{p:exttubular2}
Suppose $\pi'\E$ lies in a single tube $T$. Then $(\E,\nabla)$ is not irreducible. In fact, it contains an invariant parabolic subsheaf whose restriction to $\XX'$ also lies in $T$.
\end{prop}

\begin{proof}
We choose a different inhomogeneous tube $T'$ of the same slope $q=w'-s\in\Z$. By the previous lemma there is a unique stable parabolic sheaf $\G'\in T'$ with $\Hom_{\XX'}(\G',\pi'S_{1,0})\neq0$. Setting $\G\coloneqq\pi'_!\G'$ we have that $\G^\perp\cong\mmod\C Q$ and ${\G'}^\perp\cong\mmod\C Q'$, where $Q'$ is extended Dynkin and $Q$ is over-extended Dynkin. The vertices $0$ and $\infty$ correspond to $S_0$ and $S_\infty$ respectively.

Now, as in Lemma~\ref{l:slopeqclass}, the category of regular $\C Q'$-modules corresponds to ${\G'}^\perp\cap\CC_q$, and $\delta_{Q'}=\delta-s\partial$. Thus $\F=\pi'\E$ corresponds to a regular module $M'$ of dimension vector $h\delta_{Q'}$. Similarly, the parabolic sheaf $\E\in\G^\perp$ corresponds to a $\C Q$-module $M$ having dimension vector $\alpha_\infty+h\delta_{Q'}\in K_0(Q)$, whose restriction to $Q'$ is of course $M'$.

Next, as in the proof of Proposition~\ref{p:notubularirred}, the category of objects $(\mathcal H,s)\in\CohConn_\zeta\XX$ with $\mathcal H\in\G^\perp$ is equivalent to finite dimensional modules over some deformed preprojective algebra $\Pi^\lambda Q$ \cite{CBcw}. Thus the pair $(\E,\nabla)$ corresponds to such a module $\widetilde M$ whose restriction to $Q$ is $M$. The existence of $\widetilde M$ tells us that $\lambda\cdot(\alpha_\infty+h\delta_{Q'})=0$.

On the other hand, let $\mathcal H$ be a stable parabolic bundle contained in some homogeneous tube of slope $q$. Then $\mathcal H'\in{\G'}^\perp$ and has class $[\mathcal H']=\delta-s\partial$, so $\mathcal H\coloneqq\pi'_!\mathcal H'$ is indecomposable, lies in $\G^\perp$, and has class $\delta-s\partial$. By Lemma~\ref{l:degdimconn} we know that $\mathcal H$ admits a $\zeta$-connection $\nabla$, in which case the pair $(\mathcal H,\nabla)$ corresponds to a module for $\Pi^\lambda Q$, and hence $\lambda\cdot\delta_{Q'}=0$.

We are therefore in the situation of \cite[Theorem 8.1]{CBHu}, so we conclude that $\widetilde M$ has a proper non-zero submodule whose restriction to $Q'$ is again regular. In other words, $(\mathcal E,\nabla)$ has a proper non-zero invariant parabolic subsheaf whose restriction to $\XX'$ lies in $T$.
\end{proof}

The final case is when $\F\in\CC_q$ is supported on at least two different tubes.

\begin{prop}
\label{p:exttubular3}
Suppose $\F=\pi'\E$ contains direct summands from distinct tubes in $\CC_q$. Then $(\E,\nabla)$ is not irreducible. In fact, it contains an invariant parabolic subsheaf whose restriction again lies in $\CC_q$.
\end{prop}

\begin{proof}
We decompose $\F=\F^1\oplus\F^2$ with $\F^1$ supported on a single tube $T$, and $\F^2$ supported away from $T$. Then $\dimv\F=h\delta$, so $\dimv\F^u=h^u\delta$ by Lemma~\ref{l:dimsdistincttubes}. Since they also have slope $q=w'-s$ we must have $[\F^u]=h^u(\delta-s\partial)$.

Assume $\phi^{11}=0$. Then, as $\phi$ has rank one, we must have $\phi^{uv}=0$ for some $u\neq v$, in which case $\nabla^{uv}$ is a homomorphism of parabolic sheaves $\F^u\to\F^v(\oom_{\XX'})$. As $\F^{1,2}$ are supported on distinct tubes in $\CC_q$, we conclude that $\nabla^{uv}=0$, and thus $(\E,\nabla)$ has an invariant parabolic subsheaf $\E^v$ restricting to the direct summand $\F^v$ of $\F$.

In particular, this applies when $h^1=1$. For, following our basic strategy, we know that $\nabla^{11}$ is a logarithmic connection on $\F^1$ whose residues have diagonal entries involving $\zeta_{jq}$ and $\phi^{11}$, so by the result of Mihai, Lemma~\ref{l:conntrace},
\[ -\deg_{\PP^1}\F^1 = \sum_i\tr\Res_{a_i}\nabla^{11} = \zeta\ast[\dimv\F^1]+\tr\phi^{11}
= h^1s+\tr\phi^{11}. \]
As $\deg_{\PP^1}\F^1=-h^1s$, we see that $\phi^{11}$ is a traceless endomorphism of $F^1_0$, which is one-dimensional, as $\dimv F^1=\delta$, and hence $\phi^{11}=0$ as claimed.

Suppose instead that $h^1\geq2$ and that $\phi^{11}\neq0$. We first lift $\F^1$ to a parabolic bundle $\G$ on $\XX$ by setting $G_\infty\coloneqq\Ima(\phi^{11})$, so that $\nabla$ restricts to a $\zeta$-connection on $\G$. Now $\G$ has dimension vector $\alpha_\infty+h^1\delta$, so by Proposition~\ref{p:exttubular2} the pair $(\G,\nabla)$ contains a proper non-zero invariant parabolic subsheaf $\mathcal H$ whose restriction $\mathcal H'$ to $\XX'$ also lies in $T$.

We can therefore write $E_0=U\oplus V\oplus W$, where $F^1_0=U\oplus V$ and $F^2_0=W$, with respect to which $\phi$ decomposes as
\[ \phi = \begin{pmatrix}
\psi^{00}&\psi^{01}&\psi^{02}\\
0&\psi^{11}&\psi^{12}\\
\psi^{20}&\psi^{21}&\psi^{22}
\end{pmatrix}. \]
Note that the top left $2\times2$ submatrix is $\phi^{11}$. Now, as $\phi$ has rank at most one, either $\psi^{20}$ or $\psi^{12}$ vanishes.

In the first case, when $\psi^{20}=0$, we see that $\nabla^{21}$ restricts to a morphism $\mathcal H'\to\F^2(\oom_{\XX'})$, but these sheaves have distinct supports, so the only such map is zero. It follows that $\mathcal H$ is an invariant parabolic subsheaf of $(\E,\nabla)$.

In the second case, when $\psi^{12}=0$, we see that $\nabla^{12}$ restricts to a morphism $\F^2\to(\F^1/\mathcal H')(\oom_{\XX'})$. Again, both sheaves lie in $\CC_q$ but have distinct supports, so the only such map is zero. It follows that there is an invariant parabolic subsheaf of $(\E,\nabla)$ restricting to $\mathcal H'\oplus\F^2$.
\end{proof}

Theorem~\ref{t:exttubular} now follows from Propositions~\ref{p:exttubular1}, \ref{p:exttubular2} and \ref{p:exttubular3}.

\section{The Deligne--Simpson Problem}


The fundamental group of $\PP^1-D$ equals $\pi_1=\langle g_1,\ldots,g_k\mid g_1\cdots g_k=1\rangle$, so $\pi$-representations are given by tuples of matrices $(A_1,\ldots,A_k)$ in $\GL_n(\C)$ satisfying $A_1\cdots A_k=1$. One can naturally ask to what extent one can fix the conjugacy classes $C'_i$ of the matrices $A_i$, and the Deligne--Simpson Problem asks whether we can find an irreducible representation, so where the matrices $A_i$ have no common proper non-zero invariant subspace.

If we fix a transversal $T$ to $\Z$ in $\C$, such as $T = \{z \mid 0 \le \mathrm{Re}(z)<1\}$, then we can define the category $\BunConn_T\PP^1$ of pairs $(E,\nabla)$ consisting of a vector bundle $E$ on $\PP^1$ equipped with a logarithmic connection $\nabla\colon E\to E\otimes\Omega$, singular over $D$, and with each residue $\Res_{a_i}\nabla$ having eigenvalues in $T$.

\begin{thm}[Riemann--Hilbert Correspondence]
Taking monodromy yields an equivalence of categories
\[ \BunConn_T\PP^1 \cong \rep\pi_1 \]
In particular, $\BunConn_T\PP^1$ is an abelian length category.
\end{thm}

\begin{proof}
In this form the result can be found as Theorem 6.1 in \cite{CBipb}.
\end{proof}

Let us write $\I=\sqrt{-1}$. Then the assignment $A\mapsto\exp(-2\pi\I A)$ induces a bijection between tuples of conjugacy classes of matrices of size $n$ with eigenvalues in $T$ and tuples of conjugacy classes in $\GL_n(\C)$. Explicitly, given conjugacy classes $(C'_1,\ldots,C'_k)$ in $\GL_n(\C)$, say given by the data $\alpha\in\Gamma_\bw$ and non-zero complex numbers $\xi=(\xi_{ip})$, it arises as the image of the tuple of conjugacy classes $(C_1,\ldots,C_k)$ given by the data $\alpha,\zeta$, where $\zeta_{ip}\in T$ satisfies $\xi_{ip}=\exp(-2\pi\I\zeta_{ip})$.

We next consider the subcategory $\rep_\xi\pi_1$ consisting of those representations $(A_1,\ldots,A_k)$ satisfying $(A_i-\xi_{iw_i})\cdots(A_i-\xi_{i1})=0$. This subcategory is closed under subquotients and direct sums, so is an abelian subcategory, but it is not closed under extensions in $\rep\pi_1$.

Under the Riemann--Hilbert Correspondence this is equivalent to the abelian category $\BunConn_\zeta\PP^1$ consisting of those pairs $(E,\nabla)$ such that $(\Res_{a_i}\nabla-\zeta_{iw_i})\cdots(\Res_{a_i}\nabla-\zeta_{i1})=0$ for all $i$.

We now have the forgetful functor $F\colon\BunConn_\zeta\XX\to\BunConn_\zeta\PP^1$ sending $(\E,\nabla)$ to $(\pi\E,\nabla)$, which in turn admits fully faithful left and right adjoints. Explicitly, the left adjoint $L$ has parabolic structure given by the images
\[ E_{ip} \coloneqq \Ima(\Res_{a_i}\nabla-\zeta_{ip})\cdots(\Res_{a_i}\nabla-\zeta_{i1}) \]
whereas the right adjoint has parabolic structure given by the kernels
\[ E_{ip} \coloneqq \Ker(\Res_{a_i}\nabla-\zeta_{iw_i})\cdots(\Res_{a_i}\nabla-\zeta_{ip+1}). \]
In other words we almost have a recollement, except that the category $\BunConn_\zeta\XX$ is not abelian.

We call a pair $(\E,\nabla)$ \emph{strict} if its parabolic structure is given by the images as above, which is if and only if the counit $LF(\E,\nabla)\to(\E,\nabla)$ is an isomorphism.

We now have the following theorem, translating between the different settings above.

\begin{thm}
We fix the data $\alpha\in\Gamma_\bw$, $\xi=(\xi_{ip})$ and $\zeta=(\zeta_{ip})$ as above, yielding tuples of conjugacy classes $(C'_1,\ldots,C'_k)$ and $(C_1,\ldots,C_k)$ respectively.

(i) There exists a representation $\pi_1\to\GL_n(\C)$ with the image of $g_i$ lying in $\overline C'_i$ for all $i$ if and only if there exists $(\E,\nabla)\in\BunConn_\zeta\XX$ with $\dimv\E=\alpha$.

(ii) There exists a representation $\pi_1\to\GL_n(\C)$ with the image of $g_i$ lying in $C'_i$ for all $i$ if and only if there exists a strict $(\E,\nabla)\in\BunConn_\zeta\XX$ with $\dimv\E=\alpha$.

(iii) There exists an irreducible representation $\pi_1\to\GL_n(\C)$ with the image of $g_i$ lying in $C'_i$ for all $i$ if and only if there exists an irreducible $(\E,\nabla)\in\BunConn_\zeta\XX$.
\end{thm}

\begin{proof}
Write $\alpha=n\alpha_\ast+\sum_{i,p}n_{ip}\alpha_{ip}$. We know that $A_i\in\GL_n(\C)$ lies in $\mathcal C'_i$ if and only if
\[ \rank(A_i-\xi_{ip})\cdots(A_i-\xi_{i1}) = n_{ip}. \]
This is equivalent to the existence of a flag
\[ \C^n = V_{i0} \supseteq V_{i1} \supseteq \cdots \supseteq V_{iw_i} = 0, \quad \dim V_{ip}=n_{ip}, \]
with $(A_i-\xi_{ip})(V_{ip-1})=V_{ip}$ for all $p$.

We can similarly characterise the closure $\overline C'_i$ by saying there exists such a flag with $(A_i-\xi_{ip})(V_{ip-1})\subseteq V_{ip}$ for all $p$.

(i) By the Riemann--Hilbert Correspondence, the existence of a representation with $A_i\in\overline C'_i$ for all $i$ is equivalent to the existence of some $(E,\nabla)\in\BunConn_\xi\PP^1$ with $\Res_{a_i}\nabla\in\overline C_i$ for all $i$. By the above observation, this is equivalent to the existence of a certain flag $(E_{ip})$ in each of the fibres $E_{a_i}$, in which case $\E=(E,E_{ip})$ is a parabolic bundle on $\XX$ having dimension vector $\alpha$, and $\nabla$ is a $\zeta$-connection on $\E$.

(ii) Again, having a representation with $A_i\in C'_i$ for all $i$ is equivalent to finding some $(E,\nabla)\in\BunConn_\xi\PP^1$ with $\Res_{a_i}\nabla\in C_i$ for all $i$, which in turn is equivalent to the existence of a certain flag $(E_{ip})$ in each of the fibres $E_{a_i}$. Then $\E=(E,E_{ip})$ is a parabolic bundle on $\XX$ having dimension vector $\alpha$, and $\nabla$ is a $\zeta$-connection on $\E$ such that the pair $(\E,\nabla)$ is strict.

We observe that $(\E,\nabla)$ is precisely the image of $(E,\nabla)$ under the left adjoint.

(iii) Under the Riemann--Hilbert Correspondence, irreducible objects in $\rep_\xi\pi_1$ correspond to irreducible objects in $\BunConn_\zeta\PP^1$. The claim therefore follows provided both the forgetful functor $F$ and its left adjoint $L$ send irreducibles to irreducibles.

Now, it is clear from our explicit description of $L$ that the unit $\mathrm{id}\to FL$ is an isomorphism, and hence that $L$ is fully faithful. Also, $L$ preserves monomorphisms and the counit $LF\to\mathrm{id}$ is a monomorphism. Finally, as a right adjoint, $F$ necessarily preserves monomorphisms.

Thus, if $(E,\nabla)$ is irreducible, and $(\E,\nabla)=L(E,\nabla)$, then any non-zero subobject $(\E',\nabla')$ of $(\E,\nabla)$ maps to a non-zero subobject of $(E,\nabla)$, so equals $(E,\nabla)$. The counit then yields an inclusion $(\E,\nabla)\subseteq(\E',\nabla')$, which must therefore be an isomorphism, and hence $(\E,\nabla)=L(E,\nabla)$ is irreducible in $\BunConn_\zeta\XX$.

Conversely, if $(\E,\nabla)$ is irreducible, then it must be equal to its subobject $LF(\E,\nabla)$, so $(\E,\nabla)$ is strict. If $(E',\nabla')$ is a non-zero subobject of $(E,\nabla)\coloneqq F(\E,\nabla)$, then it is sent under $L$ to a subobject $L(E',\nabla')\subseteq(\E,\nabla)$, as $L$ preserves monomorphisms, and hence we get equality. Using that the unit is an isomorphism we see that $(E',\nabla')=(E,\nabla)$, and hence that $(E,\nabla)$ is irreducible.
\end{proof}

Let us fix the data $\alpha\in\Gamma_\bw$, $\xi=(\xi_{ip})$, and $\zeta=(\zeta_{ip})$ as above. We recall the definitions
\[ \xi^{[\alpha]} = \prod_{i,p}\xi_{ip}^{n_{ip-1}-n_{ip}} \quad\textrm{and}\quad 
\zeta\ast[\alpha] = \sum_{i,p}\zeta_{ip}(n_{ip-1}-n_{ip}). \]
Then $\exp(-2\pi\I\zeta\ast[\alpha])=\xi^{[\alpha]}$, so $\xi^{[\alpha]}=1$ if and only if $\zeta\ast[\alpha]\in\Z$.

\begin{thm}
\label{t:nosolspecialcases}
There is no irreducible solution to the Deligne--Simpson Problem in the following cases:

(a) The quiver $Q_\bw$ is extended Dynkin, $\xi^{m[\delta]}=1$, and the dimension vector is $\alpha = h m \delta$ with $h\ge 2$.

(b) The quiver $Q_\bw$ is over-extended Dynkin, $\xi^{[\alpha_\infty]}=\xi^{[\delta]}=1$, and the dimension vector is $\alpha = \alpha_\infty + h\delta$ with $h\ge 2$.
\end{thm}

\begin{proof}
By the previous theorem it is equivalent to show that there are no irreducible objects $(\E,\nabla)\in\BunConn_\zeta\XX$ in the following two situations.

(a) The quiver $Q_\bw$ is extended Dynkin, $\zeta\ast[m\delta]\in\Z$, and the dimension vector is $\alpha=hm\delta$ with $h\geq2$.

(b) The quiver $Q_\bw$ is over-extended Dynkin, $\zeta\ast[\alpha_\infty]=0$ and $\zeta\ast[\delta]\in\Z$, and the dimension vector is $\alpha = \alpha_\infty + h\delta$ with $h\ge 2$.

In (b) we have used the fact that $\xi^{[\alpha_\infty]}=1$ if and only if $\xi_{iw_i}=\xi_{iw_i-1}$, which is equivalent to saying $\zeta_{iw_i}=\zeta_{iw_i-1}$, or that $\zeta\ast[\alpha_\infty]=0$.

Now the case (a) is precisely our Proposition~\ref{p:notubularirred}, whereas case (b) is precisely our Theorem~\ref{t:exttubular}.
\end{proof}

\section{Multiplicative preprojective algebras}
\label{s:mainproof}

In Theorem~\ref{t:nosolspecialcases} we have shown that there is no irreducible solution to the Deligne--Simpson Problem in certain special cases. We use reflection functors for multiplicative preprojective algebras to complete the proof of Theorem~\ref{t:dsp}.

Let $Q$ be a quiver with vertex set $I$ and let $q \in (\C^*)^I$. Let $\Lambda^q(Q)$ be the corresponding multiplicative preprojective algebra, see \cite{CBSh}. For $\alpha\in \Gamma_Q$, say $\alpha = \sum_{i\in I} n_i \alpha_i$, we define $q^\alpha = \prod_{i\in I} q_i^{n_i}$.

The connection with the Deligne-Simpson Problem is given by the following lemma. Given conjugacy classes $C_1,\dots,C_k$ in $\GL_n(\C)$ as in the introduction, we choose $k,\bw,\xi$ as there, and obtain a quiver $Q_\bw$, root lattice $\Gamma_\bw$ and an element $\alpha_C \in \Gamma_\bw$. We define $q_C \in (\C^*)^{I_\bw}$ by $q_* = 1/\prod_{i=1}^k \xi_{i1}$ and $q_{ip} = \xi_{ip}/\xi_{i,p+1}$, so that $\xi^{[\alpha]} = 1/q^\alpha$ for all $\alpha\in \Gamma_\bw$. By \cite[Lemma 8.3]{CBSh} we have the following.

\begin{lem}
\label{l:dspmulcorr}
There is an irreducible solution to $A_1\dots A_k = 1$ with matrices $A_i\in C_i$ if and only if there is a simple $\Lambda^{q_C}(Q_\bw)$-module of dimension vector $\alpha_C$.
\end{lem}

In the rest of this section $Q$ is an arbitrary quiver with vertex set $I$ and $q\in(\C^*)^I$. For compatibility with earlier work, we now change notation: the root lattice is $\Z^I$, the simple roots are denoted $\epsilon_i$, and the components of $\alpha\in\Z^I$ are denoted $\alpha_i\in\Z$. Thus, for example, $q^\alpha = \prod_{i\in I} q^{\alpha_i}$. We write $\langle -,-\rangle$ for the Euler form as in the introduction, but now it is a bilinear form on $\Z^I$, its symmetrization is $(-,-)$ and $p(\alpha) = 1 -\langle \alpha,\alpha\rangle$.

If $v$ is a loopfree vertex in $I$, there is a reflection $s_v :\Z^I\to \Z^I$ given by $s_v(\alpha) = \alpha - (\alpha,\epsilon_v)\epsilon_v$. By \cite{CBSh} there is a multiplicative dual to $s_v$ given by 
\[
u_v : (\C^*)^I\to (\C^*)^I, 
\quad
u_v(q)_w = q_v^{-(\epsilon_v,\epsilon_w)} q_w,
\]
and a reflection functor $F_q$ giving an equivalence from the category of modules for $\Lambda^q(Q)$ to modules for $\Lambda^{u_v(q)}(Q)$, and acting as the reflection $s_v$ on dimension vectors.

Consider pairs $[q,\alpha]$ with $q\in (K^*)^I$ and $\alpha\in\Z^I$.
The reflection at a loopfree vertex $v$ is \emph{admissible} for the pair $[q,\alpha]$ if $q_v\neq 1$. 
The equivalence relation $\sim$ on the set of pairs is generated by 
$[q,\alpha]\sim[u_v(q),s_v(\alpha)]$ whenever the reflection at $v\in I$ is admissible for $[q,\alpha]$.
This is analogous to \cite[\S2]{CBmm}, except that we use multiplicative reflections.
If $[q,\alpha]\sim[q',\alpha']$, then the reflection functors ensure that there is a simple $\Lambda^{q}(Q)$-module of dimension $\alpha$ if and only if there is a simple $\Lambda^{q'}(Q)$-module of dimension $\alpha'$.

We are going to adapt \cite[\S\S5,7,8]{CBmm} to this situation, so we use bars to distinguish from the notation there.
Let $\bar R^+_q$ be the set of positive roots $\alpha$ with $q^\alpha=1$.
Let $\N \bar R^+_q$ be the set of sums of elements of $\bar R^+_q$.
Let $\bar\Sigma_q$ be the set of elements $\alpha\in \bar R^+_q$
such that there is no decomposition $\alpha = \beta+\gamma+\cdots$ with 
$p(\alpha)>p(\beta)+p(\gamma)+\cdots$ and $\beta,\gamma,\ldots\in \bar R^+_q$.

\begin{thm}
If $\alpha\in\N^I$, then $\alpha\in\bar\Sigma_q$ if and only if $0\neq \alpha\in \N \bar R^+_q$
and $(\beta,\alpha-\beta)\le -2$ whenever $\beta,\alpha-\beta$ are non-zero and in $\N \bar R^+_q$.
\end{thm}

\begin{proof}
This is analogous to \cite[Theorem 5.6]{CBmm}, using the obvious modifications of \cite[Lemmas 5.1--5.5]{CBmm}.
\end{proof}

Let $\bar F_q$ be the set of $\alpha\in \bar R^+_q$ such that $(\alpha',\epsilon_i)\le 0$ for any $[q',\alpha']\sim [q,\alpha]$ and any vertex $i$ with $q'_i=1$.

\begin{lem}
\label{l:simplecoodorfq}
If there is a simple $\Lambda^q$-module of dimension $\alpha$, then either $[q,\alpha]$
is equivalent to a pair $[q',\alpha']$ with $\alpha'$ the coordinate vector of a loopfree vertex, or $\alpha\in \bar F_q$.
\end{lem}

\begin{proof}
This is obtained by modifying \cite[Lemma 7.4]{CBmm}, using the obvious modifications of \cite[Lemma 7.1, 7.3]{CBmm}, and replacing \cite[Lemma 7.2]{CBmm} by \cite[Lemma 5.1]{CBSh}.
\end{proof}

\begin{thm}
\label{t:reducetocases}
If $[q,\alpha]$ is a pair with $\alpha\in \bar F_q \setminus \bar\Sigma_q$, then after first passing to an 
equivalent pair, and then passing to the support quiver of $\alpha$ and the corresponding restrictions of $q$ and $\alpha$,
one of the following cases holds:

(I) $Q$ is extended Dynkin with minimal positive imaginary root $\delta$, 
the dimension vector is $\alpha = h m \delta$ with $m\ge 1$ and $h\ge 2$ 
and $q^{\delta}$ a primitive $m$th root of unity.

(II) $Q$ is a disjoint union of two parts connected by one arrow between vertices $i$, $j$, with $\alpha_i = \alpha_j=1$,
and if $\alpha = \beta+\gamma$ where $\beta$ and $\gamma$ are the restriction of $\alpha$ to each
side of the edge joining $i$ and $j$, then $q^\beta = q^\gamma = 1$.

(III) $Q$ is a disjoint union of two parts connected by one arrow between vertices $j$, $k$ with $\alpha_j=1$, 
the part containing $k$ is extended Dynkin with $k$ as an extending vertex and minimal positive imaginary root $\delta$, 
the restriction of $\alpha$ to this part is $h\delta$ with $h\ge 2$ and $q^{\delta}=1$.
\end{thm}

\begin{proof}
This is an analogue of \cite[Theorem 8.1]{CBmm}, and is obtained by modifying the proof of that result, including straightforward analogues of \cite[Lemmas 8.2--8.15]{CBmm} (8.8 and 8.13 need no change). 
This result has been observed independently, and generalized, by Schedler and Tirelli \cite[Theorem 6.16]{ST}.
\end{proof}

\begin{thm}
\label{t:mainthmformultpreproj}
If there is a simple $\Lambda^q(Q)$-module of dimension $\alpha$, then $\alpha\in\bar\Sigma_q$.
\end{thm}

\begin{proof}
If $[q,\alpha]\sim[q',\alpha']$, then by the multiplicative analogue of \cite[Lemma 5.2(3)]{CBmm}, we have $\alpha\in\bar\Sigma_q$ if and only if $\alpha'\in\bar\Sigma_{q'}$. Thus by 
Lemma~\ref{l:simplecoodorfq}, if there is a simple $\Lambda^q(Q)$-module of dimension $\alpha$, we may assume either that $\alpha$ is the coordinate vector of a loopfree vertex, or else $\alpha\in\bar F_q$. In the first case we trivially have $\alpha\in\bar\Sigma_q$, so it is enough to consider the case $\alpha\in\bar F_q$, and then it suffices to show that there is no simple module in the three cases (I), (II), (III) of Theorem ~\ref{t:reducetocases}.

In case (I) there is no simple by Theorem~\ref{t:nosolspecialcases}(a).

In case (II), the conditions on $q$ ensure that the arrows from $i$ to $j$ and from $j$ to $i$ 
are represented in the simple module by $1\times 1$ matrices with product zero. Thus one of the matrices is zero, and hence the restriction of the representation to one of the parts of the quiver will be a non-trivial subrepresentation, contradicting simplicity.

In case (III), the conditions ensure that the representation restricts to give a representation of a multiplicative preprojective algebra for the full subquiver $Q''$ of $Q$ given by the extended Dynkin subquiver together with vertex $j$, and moreover this representation must also be simple. If the extended Dynkin quiver is star-shaped, this is impossible by Theorem~\ref{t:nosolspecialcases}(b). Suppose therefore that the extended Dynkin quiver is of type $\tilde{A}_n$. Since, up to isomorphism, the multiplicative preprojective algebra doesn't depend on the orientation of the quiver, we may assume that $Q''$ has shape
\[ \begin{tikzcd}
&&[-20pt] 1 \arrow[dl,"b_0"] & 2 \arrow[l,"b_1"]\\
j \arrow[r,"a"] & k \arrow[dr,"b_r"] &&&[-20pt] 3 \arrow[ul,"b_2"]\\
&& r \arrow[r] & {} \arrow[ur,dotted]
\end{tikzcd}\]
for some $r\ge 0$. Since the multiplicative preprojective algebra also doesn't depend on the choice of ordering of the arrows, the relations at the vertices $k$ and $1,\dots,r$ can be written in the form
\[
(1+aa^*)(1+b_0 b_0^*) = q_k (1+b_r^* b_r), \quad\text{and}
\]
\[
(1+b_i b_i^*) = q_i (1+b_{i-1}^* b_{i-1}) \quad (i=1,\dots,r).
\]
In the representation, the arrows $b_i$ and $b_i^*$ are all represented by $h\times h$ matrices, while $a$ is an $h\times 1$ matrix and $a^*$ is a $1\times h$ matrix. Also, the condition $q^\delta=1$ gives $q_k q_1 \cdots q_r = 1$. 

Consider the (additive) deformed preprojective algebra $\Pi^\lambda(Q'')$, where $\lambda$ is given by
\[
\lambda_j = 0, \quad \lambda_k = q_k-1, \quad \lambda_i = \frac{q_i-1}{q_1\cdots q_i} \quad (i=1,\ldots,r).
\]
Clearly $\lambda\cdot\delta=0$. Defining
\[
a^\circ = a^*(1+b_0 b_0^*), \quad
b_i^\circ = b_0^*, \quad
b_i^\circ = \frac{1}{q_1 q_2\cdots q_i} b_i^* \quad (i=1,\ldots,r)
\]
we obtain
\[
b_i b_i^\circ - b_i^\circ b_i = \lambda_i 1 \quad (i=1,\ldots,r)
\]
and
\[
b_0 b_0^\circ - b_r^\circ b_r + a a^\circ = \lambda_k 1.
\]
Summing the traces of all these, and using that $\lambda\cdot\delta=0$, we see that $a a^\circ$ has trace zero. Thus $a^\circ a$ has trace zero, and since it is a $1\times 1$ matrix, it is zero. Thus the matrices $a,b_1,\ldots,b_r$, together with $a^\circ,b_1^\circ,\ldots,b_r^\circ$ give a representation of $\Pi^\lambda(Q'')$. Moreover it is clearly simple, since any subrepresentation, by the reverse process, gives rise to a submodule of the $\Lambda^q(Q)$-module. This is however impossible by \cite[Theorem 9.1]{CBmm} or \cite[Theorem 1.1]{CBHu}.
\end{proof}

As mentioned in the introduction, the necessity part of Theorem~\ref{t:dsp} now follows from Lemma~\ref{l:dspmulcorr} and Theorem~\ref{t:mainthmformultpreproj}. 

\end{document}